\documentclass[reqno, 12pt]{amsart}
\pdfoutput=1
\makeatletter
\let\origsection=\section \def\section{\@ifstar{\origsection*}{\mysection}} 
\def\mysection{\@startsection{section}{1}\z@{.7\linespacing\@plus\linespacing}{.5\linespacing}{\normalfont\scshape\centering\S}}
\makeatother        

\usepackage{amsmath,amssymb,amsthm}
\usepackage{mathrsfs}
\usepackage{mathabx}\changenotsign
\usepackage{dsfont}
 
\usepackage{xcolor}
\usepackage[backref]{hyperref}
\hypersetup{
    colorlinks,
    linkcolor={red!60!black},
    citecolor={green!60!black},
    urlcolor={blue!60!black}
}

\usepackage[open,openlevel=2,atend]{bookmark}

\usepackage[abbrev,msc-links,backrefs]{amsrefs} 
\usepackage{doi}

\renewcommand{\PrintDOI}[1]{\doi{#1}}

\usepackage[T1]{fontenc}
\usepackage{lmodern}
\usepackage[babel]{microtype}
\usepackage[english]{babel}

\usepackage{geometry}
\numberwithin{equation}{section}
\numberwithin{figure}{section}

\usepackage{enumitem}
\def\rmlabel{\upshape({\itshape \roman*\,})}

\def\alabel{\upshape({\itshape \alph*\,})}

\def\nlabel{\upshape({\itshape \arabic*\,})}

\let\polishlcross=\l
\def\l{\ifmmode\ell\else\polishlcross\fi}

\let\emptyset=\varnothing
\let\setminus=\smallsetminus

\makeatletter
\def\moverlay{\mathpalette\mov@rlay}
\def\mov@rlay#1#2{\leavevmode\vtop{   \baselineskip\z@skip \lineskiplimit-\maxdimen
   \ialign{\hfil$\m@th#1##$\hfil\cr#2\crcr}}}
\newcommand{\charfusion}[3][\mathord]{
    #1{\ifx#1\mathop\vphantom{#2}\fi
        \mathpalette\mov@rlay{#2\cr#3}
      }
    \ifx#1\mathop\expandafter\displaylimits\fi}
\makeatother

\newcommand{\dcup}{\charfusion[\mathbin]{\cup}{\cdot}}
\newcommand{\bigdcup}{\charfusion[\mathop]{\bigcup}{\cdot}}

\DeclareFontFamily{U}  {MnSymbolC}{}
\DeclareSymbolFont{MnSyC}         {U}  {MnSymbolC}{m}{n}
\DeclareFontShape{U}{MnSymbolC}{m}{n}{
    <-6>  MnSymbolC5
   <6-7>  MnSymbolC6
   <7-8>  MnSymbolC7
   <8-9>  MnSymbolC8
   <9-10> MnSymbolC9
  <10-12> MnSymbolC10
  <12->   MnSymbolC12}{}
\DeclareMathSymbol{\powerset}{\mathord}{MnSyC}{180}
\DeclareMathSymbol{\leftY}{\mathord}{MnSyC}{42}

\DeclareSymbolFont{symbolsC}{U}{txsyc}{m}{n}
\SetSymbolFont{symbolsC}{bold}{U}{txsyc}{bx}{n}
\DeclareFontSubstitution{U}{txsyc}{m}{n}
\DeclareMathSymbol{\strictif}{\mathrel}{symbolsC}{74}

\let\epsilon=\varepsilon
\let\eps=\epsilon
\let\rho=\varrho
\let\theta=\vartheta
\let\kappa=\varkappa

\def\NN{{\mathds N}}
\def\ZZ{{\mathds Z}}
\def\RR{{\mathds R}}
\def\QQ{{\mathds Q}}

\theoremstyle{plain}
\newtheorem{thm}{Theorem}[section]
\newtheorem{theorem}[thm]{Theorem}
\newtheorem{problem}[thm]{Problem}
\newtheorem{fact}[thm]{Fact}
\newtheorem{prop}[thm]{Proposition}

\newtheorem{corollary}[thm]{Corollary}
\newtheorem{lemma}[thm]{Lemma}

\theoremstyle{definition}

\newtheorem{dfn}[thm]{Definition}

\newtheorem{conjecture}[thm]{Conjecture}

\usepackage{accents}

\let\phi=\varphi

\DeclareMathOperator*{\diam}{diam}

\newcommand{\olC}{C}
\newcommand{\fh}{\hat{f}}

\newcommand{\n}[1]{\|#1\|}
\newcommand{\Lx}{L_{\fx}}
\newcommand{\Q}{Q_{\varepsilon}(\fx)}

\newcommand{\Sx}{S(\fx)}
\newcommand{\Sy}{S(\fy)}
\newcommand{\Sa}{S(\fa)}

\newcommand{\cH}{\mathcal H}
\newcommand{\cU}{\mathcal U}
\newcommand{\cV}{\mathcal V}
\newcommand{\fx}{\mathbf x}
\newcommand{\fs}{\mathbf s}
\newcommand{\fd}{\mathbf d}
\newcommand{\fy}{\mathbf y}
\newcommand{\fa}{\mathbf a}
\newcommand{\fb}{\mathbf b}

\let\eps=\varepsilon
\def\dist{\mathrm{dist}}

\begin{document}

\author[K.~Engel]{Konrad Engel} 
\address{Universit\"at Rostock,  Institut f\"ur Mathematik, 18051 Rostock, Germany}
\email{konrad.engel@uni-rostock.de} 

\author[Th. Mitsis]{Themis Mitsis} 
\address{Department of Mathematics and Applied Mathematics, University of Crete, 
70013 Heraklion, Greece}
\email{themis.mitsis@gmail.com} 

\author[Chr. Pelekis]{Christos Pelekis} 
\address{Institute of Mathematics, Czech Academy of Sciences, \v{Z}itna 25, 
Praha 1, Czech Republic}
\thanks{Research was supported by the Czech Science Foundation, grant number 
GJ16-07822Y, by GA\v{C}R project 18-01472Y and RVO: 67985840}
\email{pelekis.chr@gmail.com} 

\author[Chr. Reiher]{Christian Reiher} 
\address{Fachbereich Mathematik, Universit\"at Hamburg, Hamburg, Germany}
\email{christian.reiher@uni-hamburg.de} 

\title[Projection inequalities for antichains]
{Projection inequalities for antichains}


\begin{abstract} 
Let $n$ be an integer with $n \ge 2$.
A set $A \subseteq \RR^n$ is called an \emph{antichain} (resp. \emph{weak antichain}) 
if it does not contain two distinct elements $\fx=(x_1,\ldots, x_n)$ and 
$\fy=(y_1,\ldots, y_n)$ satisfying $x_i\le y_i$ (resp. $x_i < y_i$) for all 
$i\in \{1,\ldots,n\}$.
We show that the Hausdorff dimension of a weak antichain $A$ in the $n$-dimensional
unit cube $[0,1]^n$ is at most $n-1$ and that the $(n-1)$-dimensional Hausdorff measure 
of $A$ is at most $n$, which are the best possible bounds. 
This result is derived as a corollary of the following {\it projection inequality},
which may be of independent interest: The $(n-1)$-dimensional Hausdorff measure of 
a (weak) antichain $A\subseteq [0, 1]^n$ cannot exceed the sum of the 
$(n-1)$-dimensional Hausdorff measures of the $n$ orthogonal projections of $A$ onto the 
facets of the unit $n$-cube containing the origin.
For the proof of this result we establish a discrete variant of the projection 
inequality applicable to weak antichains in $\ZZ^n$ and combine it with ideas 
from geometric measure theory.
\end{abstract}
\maketitle

\section{Introduction}
\label{sec:intro}

Sperner's theorem~\cite{sperner}, a cornerstone of extremal set theory, determines 
for each positive integer $n$ the maximal size of an antichain in the power set
of an $n$-element set and describes the extremal configurations. In the statement that 
follows, $[n]$ denotes the set $\{1, \ldots, n\}$ of the first $n$ natural numbers 
and $A\subseteq \powerset([n])$ is said to be an {\it antichain} if $x\not\subseteq y$
holds for any distinct $x, y\in A$. 

\begin{theorem}[Sperner] 
\label{thm:sperner}
	If $n\ge 1$ is an integer and $A\subseteq \powerset([n])$ is an antichain, 
	then $|A|\leq\binom{n}{\lfloor n/2 \rfloor}$. Equality holds if and only if 
	for some $\ell\in\bigl\{\lfloor n/2 \rfloor, \lceil n/2 \rceil\bigr\}$ the set $A$
	is the collection of all $\ell$-element subsets of $[n]$.
\end{theorem}

Sperner's fundamental result has been generalised in various ways and gave rise to 
a substantial body of future developments both within extremal set theory and 
beyond (see~\cites{Anderson, Engeltwo}).  

Observe that via characteristic vectors the power set $\powerset([n])$ can be 
identified with the set $\{0, 1\}^n$ of $n$-dimensional $0$-$1$-vectors. Moreover, 
for any two subsets $x$ and $y$ of~$[n]$ we have $x\subseteq y$ if and only if the 
characteristic vector of $x$ is coordinate-wise at most the characteristic vector of $y$. 
Therefore, Sperner's theorem can be reformulated as a statement about $\{0, 1\}^n$ 
equipped with the product partial ordering. It seems natural to ask what happens when 
one replaces $\{0, 1\}^n$ by the $n$-dimensional unit cube $[0, 1]^n$.       
   
Let us fix the following notation for discussing such situations. 
Given two $n$-tuples $\fx=(x_1,\dots,x_n)$ and $\fy=(y_1,\dots,y_n)$ in $\RR^n$, 
we write $\fx \le \fy$ if $x_i \le y_i$ for all $i \in [n]$. 
Moreover, if $\fx \le \fy$ and $\fx \ne \fy$ we write $\fx < \fy$,
while $\fx \leftY \fy$ has the stronger meaning that $x_i < y_i$ holds for all $i \in [n]$.
A set $A \subseteq \RR^n$ is called an \emph{antichain} (resp. \emph{weak antichain})
if it does not contain two elements $\fx$ and $\fy$ satisfying $\fx < \fy$ 
(resp. $\fx \leftY \fy$). So every antichain is also a weak antichain. 

In order to get some deeper insight first we replace the unit cube $[0, 1]^n$ by its discretization
$D_m^n=\bigl\{\frac{0}{m},\frac{1}{m},\dots,\frac{m-1}{m}\bigr\}^n$, where $m$ is a 
fixed positive integer. De Bruijn et al.~\cite{Bruijn_et_al} proved that the sets 
$A_{n,\ell}=\bigl\{\fx \in D_m^n\colon \sum_{i=1}^n x_i=\ell\bigr\}$, where 
$\ell \in \bigl\{\lfloor n(m-1)/2 \rfloor, \lceil n(m-1)/2 \rceil\bigr\}$, are maximum 
antichains.
Using chains of the form 
\[
	\cdots < (x_1,x_2,\dots,x_n) 
	< \bigl(x_1+\tfrac{1}{m},x_2+\tfrac{1}{m},\dots,x_n+\tfrac{1}{m}\bigr) < \cdots
\]
it is easy to show that the set $W_n=\{\fx \in D_m^n\colon x_i=0 \text{ for some } i\}$ is a maximum weak antichain. Note that $|A_{n,\ell}|=O\left(m^n/\sqrt{n}\right)$ as $n \rightarrow \infty$ and that $|W_n|=m^n-(m-1)^n$, whence $|A_{n,\ell}|/|W_n|= O(1/\sqrt{n})$ as $n \rightarrow \infty$.

Now we come back to the unit cube $[0, 1]^n$. Obviously, the set 
\[
	A_n^\star=\Bigl\{\fx \in [0,1]^n\colon \sum_{i=1}^n x_i= n/2\Bigr\}
\]
is an antichain and the set 
\[
	W_n^\star=\bigl\{\fx \in [0,1]^n\colon x_i=0 \text{ for some } i\bigr\}
\]
is a weak antichain in $[0, 1]^n$.
In view of its similarity with the previous extremal configurations 
one might expect them to have an interesting maximality property. 
Questions addressing the extremality of (weak) antichains in $[0, 1]^n$ become meaningful 
as soon as one agrees on an (outer) {\it measure} on~$[0, 1]^n$ that allows us to compare 
any two different candidates. The first measure on~$[0, 1]^n$ that usually comes to mind 
is the $n$-dimensional Lebesgue measure. However, the antichain $A_n^\star$ (and also the weak antichain $W_n^\star$) is null with respect 
to this measure and the following result due to the first author~\cite{engel} shows that, 
actually, all other antichains in $[0, 1]^n$ are null in this sense as well. 

\begin{theorem}
\label{thm:engel} 
	If $c>0$ and $A$ is a Lebesgue measurable subset of $[0,1]^n$ that does not contain 
	two elements $\fx\le\fy$ with $\sum_{i=1}^n(y_i-x_i) \ge c$, then the Lebesgue measure 
	of $A$ cannot exceed the Lebesgue measure of the optimal set
	\[ 
		A(c) =\Bigl\{\fx \in [0,1]^n\colon 
			\frac{n-c}{2}\leq \sum_{i=1}^n x_i < \frac{n+c}{2}\Bigr\}\,. 
	\]
\end{theorem}

As a matter of fact, the antichain $A_n^\star$ and the weak antichain $W_n^\star$ are not only null with respect to the 
Lebesgue measure, 
but they also have the intuitively stronger property of being $(n-1)$-dimensional. One may thus 
wonder  
\begin{enumerate}[label=\nlabel]
	\item\label{it:Q1} whether every antichain and weak antichain in $[0, 1]^n$ is at 
		most $(n-1)$-dimensional
	\item\label{it:Q2} and if so, whether $A_n^\star$ and $W_n^\star$ are in a natural sense the ``largest''  
		$(n-1)$-dimensional antichain resp. weak antichain in $[0, 1]^n$.
\end{enumerate}

The perhaps most natural measure theoretic concepts for making these questions 
precise are Hausdorff dimension and Hausdorff measure, so let us briefly recall their 
definitions. If~$U$ is a non-empty subset of $\RR^n$, we denote its {\it diameter}
by $\diam(U)$. For real numbers~$s\ge0$, $\delta>0$ and for $A\subseteq \RR^n$
we write 
\[
	\cH_{\delta}^{s}(A) 
	= 
	\alpha_s\inf 
	\left\{\sum\nolimits_{i\in\NN}\textrm{diam}(U_i)^s\colon 
		A \subseteq \bigcup\nolimits_{i\in\NN} U_i \; \text{and} \;  
		\textrm{diam}(U_i)\le \delta  
		\text{ for every } i\in\NN\right\}\,,
\]
where the normalisation factor $\alpha_s = \frac{\pi^{s/2}}{2^s\Gamma(s/2+1)}$ 
denotes the volume of the $s$-dimensional sphere of radius $\frac{1}{2}$. Its presence  
ensures that if $s=n$ and $\{U_i\colon i\in\NN\}$ is a collection of mutually 
disjoint balls, then the right side agrees with the total volume of these balls. 
For later use we remark that these quantities $\cH_{\delta}^{s}(A)$ are very robust 
under the addition of various regularity properties that can be imposed on the sets $U_i$. 
For instance, one could insist that these sets need to be closed 
(see e.g.~\cite{Bishop_Peres}*{p.4}).   

Evidently for fixed $s$ and $A$, the value of $\cH_{\delta}^{s}(A)$ increases
as $\delta$ decreases and thus the limit 
\[
	\cH^{s}(A)=\lim_{\delta\rightarrow 0}\cH_{\delta}^{s}(A)\,,
\]
called the $s$-\emph{dimensional Hausdorff measure of $A$}, exists. 
It is well-known that $\cH^n$ agrees on $\RR^n$ with the $n$-dimensional 
Lebesgue outer measure (see e.g.~\cite{Evans_Gariepy}*{p.87}).
In particular, $\cH^{n}([0,1]^n)=1$.

The \emph{Hausdorff dimension} of $A$, denoted by $\dim_H A$, is defined by
\[
	\dim_H(A) = \inf \left\{ s \colon \cH^s(A) = 0   \right\}\,.
\]
One checks easily that $s<\dim_H(A)$ implies $\cH^s(A)=\infty$, while for 
$s>\dim_H(A)$ one has $\cH^s(A)=0$. Therefore, for fixed $A$ the only value 
of $s$ for which $\cH^s(A)$ can have a ``non-trivial'' value is $s=\dim_H(A)$.
We refer to~\cites{Bishop_Peres, Evans_Gariepy, Falconer_1990} for legible textbooks 
on the topic. 

Let us now return to our problems~\ref{it:Q1} and~\ref{it:Q2}. Our main result does indeed
imply that every weak antichain (and hence also every antichain) $A\subseteq [0, 1]^n$ satisfies $\dim_H(A)\le n-1$. The second 
question, however, has a negative answer concerning antichains, but a positive answer concerning weak antichains. 
Notice that $W_n^\star$ is a union of $n$ facets of the unit $n$-cube, wherefore $\cH^{n-1}(W_n^\star)=n$.   

In order to prove the optimality of $W_n^\star$ we establish a more general result,
which we call the {\it projection inequality}. 
It asserts that the $\cH^{n-1}$-measure of a weak antichain $A\subseteq [0, 1]^n$ 
is at most the sum of the $\cH^{n-1}$-measures of the orthogonal projections of $A$
to the $n$ facets of the unit cube containing the origin.   
Such projections are just deleting a fixed coordinate. When $n$ is clear from 
the context and $i\in [n]$ we write $\pi_i\colon \RR^n\longrightarrow \RR^{n-1}$
for the projection defined by
\[
	\pi_i(x_1,\dots,x_n)=(x_1,\dots,x_{i-1},x_{i+1},\dots,x_n)\,.
\]

Our main result on weak antichains in the unit $n$-cube reads as follows. 

\begin{theorem}
\label{thm:main}
	If $A$ is a weak antichain in $[0,1]^n$, then
	\[
		\cH^{n-1}(A) \le \sum_{i=1}^n\cH^{n-1}\bigl(\pi_i(A)\bigr)\,.
	\]
	In particular, 
	\begin{equation}\label{eq:main}
		\cH^{n-1}(A)\le n
				\quad\text{ and }\quad
		\dim_H(A)\le n-1 \,.
	\end{equation}
\end{theorem}

In order to find ``better'' antichains than $A_n^\star$ one can start with the ``best'' weak antichain $W_n^\star$ and deform it slightly to obtain an antichain.
This can be done in a polyhedral way, but here we present a ``smooth'' way:
Consider the hypersurface 
\[
	A_p=\left\{\fx \in [0,1]^n\colon \|\fx\|_p=1\right\}
\]
as $p\to\infty$. One can easily verify that the $\ell^p$-unit ball 
$B_p=\{\fx\in\RR^n\colon\|\fx\|_p\leq1\}$ converges with respect 
to the Hausdorff metric to the $\ell^\infty$-unit ball as $p\to\infty$. 
Moreover, it is well-known (\cite{schneider}*{p.219}) that if a sequence of convex 
bodies $K_i$ converges to a convex body~$K$ with respect to the Hausdorff metric, 
then $\cH^{n-1}(\partial K_i)\to\cH^{n-1}(\partial K)$. For these reasons 
we have $\cH^{n-1}(A_p)\to n$ as $p\to \infty$.

\begin{corollary} \label{cor:sup}
	For every positive integer $n$ we have
	\[
		\sup\bigl\{\cH^{n-1}(A)\colon A\subseteq [0, 1]^n \text{ is an antichain}\bigr\}=n\,.
	\]
\end{corollary}

Note that $\cH^1(A_2^\star)=\sqrt{2}<2$ and
$\cH^2(A_3^\star)=3\sqrt{3}/4<3$, and in general one can easily check
that~$\cH^{n-1}(A_n^\star)<n$ and thus $A_n^\star$ is indeed not optimal.

Obviously, for every antichain $A\subseteq \RR^n$, the restriction of $\pi_i$ to $A$ is 
injective. We emphasise, however, that the projection inequality does not remain true if the 
antichain-condition is relaxed to an injectivity-condition. Indeed, Foran \cite{Foran}*{p.813} 
constructed bijective functions from $[0,1]$ onto $[0,1]$ whose graphs have arbitrary large 
$\cH^1$-measure.

One of the two central ideas in our approach is to discretise Theorem~\ref{thm:main}. 
In particular, a basic tool for  the proof of Theorem~\ref{thm:main} is the following 
discrete variant of the projection inequality replacing $\cH^{n-1}$ by the counting measure. 

\begin{theorem}
\label{thm:discrete}
	If $A$ is a finite weak antichain in $\ZZ^n$, then 
	\[
		|A| \le \sum_{i=1}^{n} |\pi_i(A)|\,.
	\]
\end{theorem}

Our article is motivated by the idea that several combinatorial statements have 
continuous counterparts. This is a rather old idea, which dates back at least to 
the 70s, and since its conception many results have been reported in a ``measurable'' 
setting (see e.g.~\cites{Bollobas, Bollobas_Varopoulos,  engel, katonaone, katonatwo}) 
or in a ``vector space'' setting (see e.g.~\cites{FranklWilson, klainrota}).

\subsection*{Organisation}
In Section~\ref{sec:discrete} we prove Theorem~\ref{thm:discrete}.
We use this result together with some ideas pertaining to geometric measure theory 
in Section~\ref{sec:anti} for proving the special case of Theorem~\ref{thm:main}  
where $A$ is an antichain. This section might be the technically most demanding part
of the article and we defer an outline of the argument to Subsection~\ref{subsec:outline}.  
In Section~\ref{sec:weak} we complete the proof of Theorem~\ref{thm:main} by reducing it
to the special case that $A$ is an antichain. The concluding remarks in Section~\ref{sec:conc}
describe several problems for future research.

\section{The discrete projection inequality}
\label{sec:discrete}

The key observation on which the proof of Theorem~\ref{thm:discrete} relies is 
that every weak antichain in~$\ZZ^n$ can be partitioned into $n$ parts 
such that for each $i\in[n]$ the projection~$\pi_i(\cdot)$ is injective on the $i$-th
part. 

\begin{proof}[Proof of Theorem~\ref{thm:discrete}]
	Let $A\subseteq \ZZ^n$ be a weak antichain. Define recursively disjoint subsets 
	$A_1, \ldots, A_{n}$ of $A$ as follows. If for some $i\in [n]$ the sets 
	$A_1, \ldots, A_{i-1}$ have just been constructed, set 
	$B_i=A\smallsetminus (A_1\cup\ldots\cup A_{i-1})$ and let $A_i$ be the set of points
	in~$B_i$ whose $i$-th coordinates are minimal. In other words, 
	$A_i$ is defined so as to contain those points $(x_1, \ldots, x_n)\in B_i$ for which 
	there exists no integer $y_i<x_i$ such that 
	$(x_1, \ldots, x_{i-1}, y_i, x_{i+1}, \ldots, x_n)\in B_i$. 
	
	We contend that the set $B_{n+1}=A\smallsetminus (A_1\cup\ldots\cup A_{n})$
	is empty. Indeed, assume indirectly that some point 
	$\fx=(x_1, \ldots, x_n)$ belongs to $B_{n+1}$. Using the definitions of 
	$A_n, \ldots, A_1$ in this order we find integers $y_n<x_n, \ldots, y_1<x_1$ 
	such that $(x_1, \ldots, x_{i-1}, y_i, \ldots, y_n)\in A_i$ for every $i \in [n]$. 
	In particular, the point $\fy=(y_1, \ldots, y_n)$ is in $A$ and satisfies 
	$\fy \leftY \fx$, contrary to~$A$ being a weak antichain. 
	This proves that $A=A_1\dcup A_2\dcup\ldots\dcup A_n$ is a partition.
	
	Now obviously for every $i \in [n]$ the projection $\pi_i$ is injective on $A_i$, 
	whence  
	\[
		|A_i| = |\pi_i(A_i)|\le |\pi_i(A)|\,.
	\]
	We conclude that 
	\[
		|A|=\sum_{i=1}^n |A_i|\le \sum_{i=1}^n|\pi_i(A)|\,,
	\]
	as desired. 
\end{proof}

\section{Antichains}
\label{sec:anti} 

\subsection{Overview}
\label{subsec:outline}

This section deals with a special case of our main theorem, where rather than 
weak antichains we only consider antichains. So explicitly we aim at the following 
result.   

\begin{prop}
\label{prop:anti}
If $A$ is an antichain in $[0,1]^n$, then
\[
	\cH^{n-1}(A)\leq\sum_{i=1}^n\cH^{n-1}\bigl(\pi_i(A)\bigr)\,.
\]
\end{prop} 

As the proof of this estimate is somewhat involved, we would like to devote the present
subsection to a discussion of our basic strategy. An obvious approach is the following 
discretisation procedure: Take a large natural number $m$, cut the unit cube into 
$m^n$ smaller cubes of side length $\frac 1m$ and keep track which of these cubes 
intersect $A$. This situation can be encoded as a weak antichain in $[m]^n$, to 
which the discrete projection inequality (Theorem~\ref{thm:discrete}) applies. 
On first sight one might hope that in the limit $m\to\infty$ this argument would yield 
Proposition~\ref{prop:anti}. But when working out the details, one discovers that 
one looses a constant factor which depends on the dimension $n$, but not on the antichain~$A$ 
itself.    

\begin{lemma}
\label{lem:D}
For every positive integer $n$ there exists a constant $D>0$ such that every weak antichain  
$A\subseteq [0,1]^n$ satisfies
\[ 
	\cH^{n-1}(A) \le D\cdot \sum_{i=1}^{n}\cH^{n-1}\bigl(\pi_i(A)\bigr) \,. 
\]
\end{lemma}

To aid the reader's orientation we remark that in Subsection~\ref{subsec:D} we will 
show this estimate for  
\begin{equation} \label{eq:D-intro}
	D=n^{(n-1)/2} \cdot \alpha_{n-1}
	 =\sqrt{\frac 2{\pi^2 n}}\cdot \left(\frac{\pi e}{2}\right)^{n/2}(1+o(1))\,,
\end{equation}
but the precise value of $D$ will be rather immaterial to what follows. 

A completely different and more analytical approach to Proposition~\ref{prop:anti} 
starts from the following observation. Suppose that 
$f\colon [0, 1]^{n-1}\longrightarrow [0, 1]$ is decreasing in each coordinate
and sufficiently smooth, and that we want to study the antichain 
\[
	A=\bigl\{(\fx,f(\fx))\colon \fx\in [0, 1]^{n-1}\bigr\}\,.
\]
Denoting the partial derivatives of $f$ by $D_1f, \ldots, D_{n-1}f$ one checks 
easily that 
\begin{align*}
	\cH^{n-1}(A) &= \int_{[0, 1]^{n-1}} 
		\sqrt{1+|D_1f(\fx)|^2+\ldots+|D_{n-1}f(\fx)|^2}\, \mathrm{d} \fx \\
		\intertext{ and }
	\cH^{n-1}\bigl(\pi_i(A)\bigr) &= \int_{[0, 1]^{n-1}} \big|D_if(\fx)\big|\mathrm{d} \fx 
		\qquad \text{ for all } i\in[n-1],
\end{align*}
for which reason the projection inequality for $A$ follows from 
\[
	\sqrt{1+|D_1f(\fx)|^2+\ldots+|D_{n-1}f(\fx)|^2}
	\le
	1+|D_1f(\fx)|+\ldots+|D_{n-1}f(\fx)|
\]
and from $\pi_n(A)=[0, 1]^{n-1}$. 

In general we need to look at antichains of the form 
\begin{equation}\label{eq:Bf}
	A=\bigl\{(\fx,f(\fx))\colon \fx\in B\bigr\}\,,
\end{equation}
where $B=\pi_n(A)\subseteq [0, 1]^{n-1}$ is arbitrary and $f\colon B\longrightarrow [0, 1]$
is only known to be decreasing in each coordinate, but not necessarily smooth. The question 
to what extent arguments that work well for smooth functions can be extended to more 
general scenarios lies at the very heart of a mathematical area known as {\it geometric 
measure theory}. In Subsection~\ref{subsec:gmt} we shall use some methods from this subject
in order to generalise the previous idea as follows.

\begin{lemma}\label{lem:gmt}
	Given an antichain $A\subseteq [0,1]^n$ and $\delta>0$ there exists a Borel set 
	${B\subseteq [0,1]^{n-1}}$ such that 
	\begin{enumerate}[label=\rmlabel]
		\item\label{it:gmt1} $\cH^{n-1}(B)>1-\delta$
		\item\label{it:gmt2} and the set $A'=A\cap\pi_{n}^{-1}(B)$ satisfies 
			$\cH^{n-1}(A')\le \sum_{i=1}^{n}\cH^{n-1}\bigl(\pi_i(A')\bigr)+\delta. $
	\end{enumerate} 
\end{lemma}

Now roughly speaking one may hope to prove Proposition~\ref{prop:anti} (up to an 
arbitrarily small additive error) by using Lemma~\ref{lem:gmt} in each coordinate 
direction, each time cutting out a substantial piece of $A$ whose $\cH^{n-1}$-measure
can be estimated quite efficiently by part~\ref{it:gmt2}. There will remain a ``small''
left-over part of~$A$, which can then be handled by means of Lemma~\ref{lem:D}. 

There is one final technical hurdle one needs to overcome when pursuing such a plan.     
The problem is that, in general, Hausdorff measure is only known to be subadditive.
So when one attempts to prove the projection inequality for an antichain $A$ by splitting 
it into two pieces, handling both pieces separately, and adding up the results, 
one can get into trouble with the right sides. We shall use the following lemma
for getting around this difficulty. 
  
\begin{lemma}
\label{lem:decomp}
	Let $A\subseteq [0,1]^n$ be an antichain and let $[0,1]^{n-1}= B'\dcup B''$ be a 
	partition of the $(n-1)$-dimensional unit cube into Borel sets. 
	Setting $A'=A\cap \pi_{n}^{-1}(B')$ and $A''=A\cap \pi_{n}^{-1}(B'')$ we have 
	\[
		\cH^{n-1}\bigl(\pi_i(A)\bigr) 
		= 
		\cH^{n-1}\bigl(\pi_i(A')\bigr) + \cH^{n-1}\bigl(\pi_i(A'')\bigr)
	\]
	for every $i\in [n]$.
\end{lemma}

We conclude the present subsection by showing that our three lemmas do indeed 
imply Proposition~\ref{prop:anti}. The proofs of the lemmas themselves are deferred 
to the three subsections that follow. But it will be convenient to prove 
Lemma~\ref{lem:decomp} in Subsection~\ref{subsec:decomp} before we turn our 
attention to Lemma~\ref{lem:gmt} in Subsection~\ref{subsec:gmt}.

\begin{proof}[Proof of Proposition~\ref{prop:anti} assuming Lemma~\ref{lem:D}--\ref{lem:decomp}]
	Fix a dimension $n\ge 1$ as well as some $\delta>0$ and let $D$ be the number 
	provided by Lemma~\ref{lem:D}. We call a subset $I\subseteq [n]$ \emph{good} 
	if every antichain $A\subseteq [0,1]^n$ with $\cH^{n-1}(\pi_i(A))\le \delta$ 
	for all $i\in I$ satisfies 
	\[
		\cH^{n-1}(A) \le \sum_{i=1}^{n} \cH^{n-1}\bigl(\pi_i(A)\bigr) + (nD+n-|I|)\delta\,.
	\]
	As a consequence of Lemma~\ref{lem:D} we know that the set $[n]$ is good. 
	Hence there exists a minimal good set $I\subseteq [n]$. 
	Assume for the sake of contradiction that $I\ne\varnothing$. 
	By permuting the coordinates if necessary we may suppose that $n\in I$. 
	Consider any antichain $A\subseteq [0,1]^n$ with 
	\begin{equation} \label{eq:I-n-good}
		\cH^{n-1}\bigl(\pi_i(A)\bigr)\le \delta 
		\quad \text{ for all } i\in I\setminus \{n\}\,.
	\end{equation} 
	By Lemma~\ref{lem:gmt} there exists a Borel set $B'$ with $\cH^{n-1}(B')>1-\delta$ 
	such that the set $A'= A\cap \pi_{n}^{-1}(B')$ satisfies 
	\begin{equation}\label{eq:43}
		\cH^{n-1}(A') \le \sum_{i=1}^{n} \cH^{n-1}\bigl(\pi_i(A')\bigr) + \delta\,. 
	\end{equation}
	Set $B''= [0,1]^{n-1}\setminus B'$ and $A''=A\cap\pi_{n}^{-1}(B'')$. 
	Notice that $\cH^{n-1}(\pi_i(A''))\le \delta$ for all $i\in I$. 
	This is because for $i\in I\setminus \{n\}$ we have
	\[ 
		\cH^{n-1}\bigl(\pi_i(A'')\bigr) \le \cH^{n-1}\bigl(\pi_i(A)\bigr)\le \delta\,, 
	\]
	while for $i=n$ the first property of $B'$ yields  
	\[
		\cH^{n-1}\bigl(\pi_n(A'')\bigr) 
		\le 
		\cH^{n-1}(B'')=1-\cH^{n-1}(B') \le \delta\,.
	\]
	Since $I$ is good, it follows that 
	\begin{equation*}
		\cH^{n-1}(A'') \le \sum_{i=1}^{n}\cH^{n-1}\bigl(\pi_i(A'')\bigr) + (nD+n-|I|)\delta\,.
	\end{equation*} 
	Adding~\eqref{eq:43} and taking Lemma~\ref{lem:decomp} into account, we deduce 
	\begin{eqnarray*}
		\cH^{n-1}(A) &\le& \cH^{n-1}(A') + \cH^{n-1}(A'') \\
		&\le& 
		\sum_{i=1}^{n} \left( \cH^{n-1}\bigl(\pi_i(A')\bigr) 
			+ \cH^{n-1}\bigl(\pi_i(A'')\bigr) \right) 
			+ (nD+n-|I|)\delta+\delta \\
		&=& 
			\sum_{i=1}^{n}\cH^{n-1}\bigl(\pi_i(A)\bigr) + (nD+n-|I\setminus\{n\}|)\delta\,.
	\end{eqnarray*}
	As this argument applies to any antichain $A$ with~\eqref{eq:I-n-good}, we have thereby
	shown that the set $I\setminus\{n\}$ is good as well, contrary to the minimality 
	of $I$.
	
	This contradiction proves that $I=\varnothing$ is good, or in other words that for 
	every antichain $A\subseteq [0,1]^n$ the estimate
	\[
		\cH^{n-1}(A) 
		\le 
		\sum_{i=1}^{n} \cH^{n-1}\bigl(\pi_i(A)\bigr) + n(D+1)\delta
	\]
	holds. In the limit $\delta\to 0$ this proves Proposition~\ref{prop:anti}.
\end{proof}

\subsection{Discretisation} 
\label{subsec:D} 

This subsection deals with the proof of Lemma~\ref{lem:D}. Let us recall the following   
well-known continuity  property of the Hausdorff measure
(see~\cite{Mattila}*{Theorem~1.4}, the remarks that follow, and~\cite{Mattila}*{Corollary 4.5}). 

\begin{fact}\label{fact:haus-mono}
	If $m\ge k\ge 1$ are integers, $E_1\subseteq E_2\subseteq \dots \subseteq \RR^m$ 
	are arbitrary sets, and $E=\bigcup_{\ell\ge 1} E_{\ell}$, then 
	\begin{align*}\tag*{$\Box$}
		\hfill \cH^k(E) = \lim_{\ell\to\infty} \cH^k(E_{\ell})\,. 
	\end{align*}
\end{fact}

The discretisation procedure by means of which we shall establish Lemma~\ref{lem:D}
works as follows. Fix a large positive integer $m$ the size of 
which determines how fine our approximation is going to be. Let 
$[0, 1]=I_1\dcup \ldots\dcup I_m$ be a partition of the unit interval into $m$ 
consecutive intervals of length $\frac 1m$. It is somewhat immaterial how we proceed with 
the boundary points $\frac 1m, \ldots, \frac{m-1}m$, but for definiteness we set   
\[
	I_{j}=
	\begin{cases}
		\bigl[\frac{j-1}{m},\frac{j}{m}\bigr) &\text{ if } j \in [m-1],\\
		\bigl[\frac{m-1}{m},1\bigr] &\text{ if } j = m.
	\end{cases}
\]
This gives rise to the partition 
\[
[0,1]^n=\bigdcup_{\mathbf{d} \in [m]^n} C(\mathbf{d})
\]
of the $n$-dimensional unit cube into $m^n$ subcubes defined by 
\[
	C(\mathbf{d})=I_{d_1}\times \dots \times I_{d_n}
\]
for every $\fd=(d_1, \ldots, d_n)\in [m]^n$.

For any subset $W\subseteq [0, 1]^n$ (not necessarily an antichain) we set
\[
	G_m(W)=\bigl\{\mathbf{d} \in [m]^n\colon \olC(\mathbf{d}) \cap W \neq \emptyset\bigr\}
\]
and observe that the disjoint union
\begin{equation} \label{eq:nest}
	H_m(W)= \bigdcup_{\mathbf{d} \in G_m(W)} \olC(\mathbf{d})
\end{equation}
is a superset of $W$. Conversely, every point in $H_m(W)$ has at most the distance 
$\sqrt{n}/m$, the common diameter of our small cubes, from an appropriate point in $W$.

Let us introduce some useful notation for such situations. 
Given $S\subseteq \RR^n$ and a point $\fx\in\RR^n$, 
we set $\dist(\fx,S) = \inf\{\|\fx-\mathbf{s}\|\colon \fs\in S\}$, 
where~$\|\cdot\|$ denotes the Euclidean norm. 
For a given positive real number $\delta$ and $S\subseteq \RR^n$
the {\it $\delta$-neighbourhood} of $S$ is defined by 
\begin{equation}\label{eq:Sdelta}
	S^{(\delta)} = \{\fx\in\RR^n \colon\dist(\fx,S) \le \delta\}\,. 
\end{equation}

Summarising the above discussion, we have 
\[
	W\subseteq H_m(W)\subseteq W^{(\sqrt{n}/m)}
\]
for every $W\subseteq [0, 1]^n$ and every $m\in \NN$. In the special case where $W$
is closed we have $W=\bigcap_{m\ge 1} W^{(\sqrt{n}/m)}$ and, consequently, 
Fact~\ref{fact:haus-mono} yields 
\begin{equation}\label{eq:closed}
	\lim_{m\to\infty} \frac{|G_m(W)|}{m^n} = \cH^n(W)\,.
\end{equation}

\begin{lemma}\label{lem:Dconv}
	If $A_1,\ldots,A_n\subseteq [0,1]^{n-1}$ are closed 
	sets and $A\subseteq [0,1]^n$ is a weak antichain with $\pi_i(A)\subseteq A_i$ for 
	all~$i\in [n]$, then 
	\[
		\cH^{n-1}(A) \le D\cdot \sum_{i=1}^{n}\cH^{n-1}(A_i)\,,
	\]
	where $D$ denotes the constant introduced in~\eqref{eq:D-intro}.
\end{lemma}

\begin{proof}
	Fix $\delta>0$, consider an arbitrary positive integer $m\ge \sqrt{n}/\delta$, 
	and set $B=G_m(A)$. The covering $A\subseteq \bigcup_{\mathbf{d} \in B} \olC(\mathbf{d})$ 
	of $A$ uses $|B|$ cubes of diameter $\sqrt{n}/m\le \delta$ and thus we have 
	\[
		\cH^{n-1}_{\delta}(A) 
		\le
		\alpha_{n-1} \cdot |B| \cdot \left(\frac{\sqrt{n}}{m}\right)^{n-1}
		\overset{\text{\eqref{eq:D-intro}}}{=}
		\frac{D|B|}{m^{n-1}}\,.
	\]

	Since $A$ is a weak antichain in $[0, 1]^n$, the set $B$ is a weak antichain in $\ZZ^n$
	and Theorem~\ref{thm:discrete} discloses
	\[
		|B| \le \sum_{i=1}^{n}|\pi_i(B)|\,.
	\]
	Obviously, $\pi_i(B)\subseteq G_m(A_i)$ for every $i\in [n]$ (where the operator 
	$G_m$ is applied to the $(n-1)$-dimensional unit cube). Hence 
	\[
		\cH^{n-1}_{\delta}(A) \le
		D\cdot\sum_{i=1}^{n} \frac{|\pi_i(B)|}{m^{n-1}} 
		\le D\cdot\sum_{i=1}^{n} \frac{|G_m(A_i)|}{m^{n-1}}
	\]
	and by~\eqref{eq:closed} we obtain in the limit $m\to\infty$ that 
	\[ 
		\cH^{n-1}_{\delta}(A) \le D\cdot \sum_{i=1}^{n}\cH^{n-1}(A_i)\,. 
	\]
	Finally $\delta\to 0$ yields the desired result. 
\end{proof}

Now a standard application of Fact~\ref{fact:haus-mono} leads to Lemma~\ref{lem:D}.

\begin{proof}[Proof of Lemma~\ref{lem:D}]
	Let $\varepsilon>0$ be arbitrary. 
	Pick for every $i\in [n]$ a sequence $(C_{i,k})_{k\ge 1}$ 
	of closed subsets of $[0,1]^{n-1}$  
	such that 
	\begin{equation} \label{eq:Cik}
		\sum_{k}\cH^{n-1}(C_{i,k})
		\le 
		\cH^{n-1}\bigl(\pi_i(A)\bigr) + \varepsilon 
	\end{equation}
	and $\pi_i(A)\subseteq \bigcup_k C_{i,k}$. 

	Now consider an arbitrary $\ell\ge 1$. Setting 
	\[
		A_{i}^{(\ell)} = \bigcup_{k\in [\ell]} C_{i,k}
	\]
	for every $i\in [n]$ and
	\[
		A^{(\ell)} = \bigcap_{i\in [n]}\pi_i^{-1}\bigl(A_{i}^{(\ell)}\bigr)
	\]
	we deduce from Lemma~\ref{lem:Dconv} that 
	\[
		\cH^{n-1}(A\cap A^{(\ell)})\le D\cdot\sum_{i=1}^{n} \cH^{n-1}\bigl(A_i^{(\ell)}\bigr)\,.
	\]
 	Using $A\subseteq \bigcup_{\ell\ge 1} A^{(\ell)}$, Fact~\ref{fact:haus-mono},
	and~\eqref{eq:Cik} we obtain in the limit $\ell\to\infty$ that 
	\[
		\cH^{n-1}(A) \le D\cdot\sum_{i=1}^{n} \cH^{n-1}\bigl(\pi_i(A)\bigr) + nD\varepsilon\,.
	\]
	As $\eps>0$ was arbitrary, this proves Lemma~\ref{lem:D}.
\end{proof}

\begin{corollary} \label{cor:N}
	If $A\subseteq [0, 1]^n$ is a weak antichain, then $\cH^{n-1}(A)\le Dn$ is finite
	and, consequently, $\cH^{n}(A)=0$. \hfill $\Box$
\end{corollary}

\subsection{The decomposition Lemma} 
\label{subsec:decomp}

Let an $n$-dimensional antichain $A\subseteq [0,1]^n$ be given, where $n\ge 2$. 
Since no two points in $A$ can agree 
in their first $n-1$ coordinates, there exists a function 
$f_A\colon \pi_n(A)\longrightarrow [0,1]$ such that
\[
	A=\bigl\{ (\fx, f_A(\fx))\colon \fx\in\pi_n(A) \bigr\}\,.
\]

The fact that $A$ is indeed an antichain is equivalent to $f(\fx)>f(\fy)$
whenever $\fx < \fy$ are in~$\pi_n(A)$. It is often convenient to extend 
this function $f_A$ in a monotonicity preserving way to the whole $(n-1)$-dimensional 
unit cube. To this end one defines $\fh_A\colon [0,1]^{n-1}\longrightarrow [0,1]$ by 
\begin{equation}\label{eq:fahut}
	\fh_A(\fx) = \inf \bigl\{f_A(\fa)\colon \fa\in\pi_n(A) \text{ and } \fa\le \fx\bigr\}
\end{equation}
for all $\fx\in [0, 1]^{n-1}$, where, in this context, $\inf(\emptyset) = 1$.
By the aforementioned fact on $f_A$ we have $\fh_A(\fx)=f_A(\fx)$ for all
$\fx\in\pi_n(A)$. Moreover, if $\fx\le\fy$ are in $[0, 1]^{n-1}$, 
then $f(\fx)\ge f(\fy)$. We shall refer to $\fh_A$ as the \emph{function 
associated with the antichain} $A$.

More generally, we call a function $f\colon [0,1]^{n-1}\longrightarrow [0,1]$ 
\emph{order-reversing},
if we have $f(\fx)\ge f(\fy)$ whenever $\fx\le \fy$. So for instance the function $\fh_A$
associated with an antichain~$A$ has just been observed to be order-reversing. We will need
the following properties of such functions proved in~\cite{Chabrillac_Crouzeix}.

\begin{lemma}
\label{lem:CC}
	If $f\colon[0,1]^{n-1}\longrightarrow [0,1]$ is order-reversing, then it is measurable 
	in the sense that preimages of Borel sets are Lebesgue measurable. 
	Moreover, $f$ is almost everywhere differentiable. 
\end{lemma}

Let us record an easy consequence. 

\begin{fact} \label{fact:L}
	Given an antichain $A\subseteq [0, 1]$ with associated function 
	$\fh_A\colon [0,1]^{n-1}\longrightarrow [0,1]$ and $c\in [0, 1]$, the set 
	\[
		L =\bigl\{ (x_1,\ldots,x_{n-1})\in [0,1]^{n-1}\colon
			 x_{n-1}<\fh_A(x_1,\ldots,x_{n-2},c)\bigr\}
	\]
	is Lebesgue measurable.
\end{fact}
	
\begin{proof}
	By Lemma~\ref{lem:CC} it suffices to check that the 
	characteristic function $\mathbf{1}_L$ of $L$ is order-reversing. 
	So let $(\fx,x_{n-1}) \le (\fy,y_{n-1})$ be given, where $\fx,\fy\in [0,1]^{n-2}$.
	We need to prove that
	\[
		\mathbf{1}_L(\fx,x_{n-1}) \ge \mathbf{1}_L(\fy,y_{n-1})\,.
	\]
	If $(\fy,y_{n-1})\notin L$ this is clear, so we may suppose that $(\fy,y_{n-1})\in L$. 
	Since $\fh_A$ is order-reversing, it follows that 
	\[ 
		x_{n-1}\le y_{n-1} < \fh_A(\fy,c)\le \fh_A(\fx,c)\, ,
	\]
	which in turn implies that $(\fx,x_{n-1}) \in L$.  
\end{proof}	

Now we are ready for the proof of Lemma~\ref{lem:decomp}, which will occupy the remainder of 
this subsection.  

\begin{proof}[Proof of Lemma~\ref{lem:decomp}]
	For $i=n$ this follows from the Lebesgue measurability of the 
	set $B'$ and from the fact that in $\RR^{n-1}$ the $(n-1)$-dimensional 
	Hausdorff outer  measure coincides with the $(n-1)$-dimensional Lebesgue outer measure.
	So without loss of generality we may henceforth assume that $i=n-1$. 
	
	Define the set-function $\nu\colon\powerset([0,1]^{n-1})\longrightarrow [0,1]$ by setting 
	\[
		\nu(E) = \cH^{n-1}\bigl(\pi_{n-1}(\pi_n^{-1}(E)\cap A)\bigr)
	\]
	for every $E \subseteq [0,1]^{n-1}$. 
	In other words, if $\fh_A\colon[0,1]^{n-1}\longrightarrow [0,1]$ is the function associated 
	with the antichain $A$, then $\nu(E)=\cH^{n-1}(F_E)$, where 
	\[
		F_E = \bigl\{(x_1,\ldots,x_{n-2}, \fh_A(x_1,\ldots,x_{n-1}))\colon 
			(x_1,\ldots,x_{n-1})\in E\cap \pi_n(A) \bigr\}\,.
	\]
	Notice that $\nu$ is an outer measure. We will show later that $B'$ is $\nu$-measurable. 
	This will imply that
	\[
		\nu\bigl(\pi_n(A)\bigr) 
		= 
		\nu\bigl(\pi_n(A)\cap B'\bigr) + \nu\bigl(\pi_n(A)\cap B''\bigr)\,,
	\]
	which is equivalent to 
	\[ 
		\cH^{n-1}\bigl(\pi_{n-1}(A)\bigr) 
		= 
		\cH^{n-1}\bigl(\pi_{n-1}(A')\bigr) + \cH^{n-1}\bigl(\pi_{n-1}(A'')\bigr)\,, 
	\]
	and the result will follow. 

	Thus it remains to show that all Borel sets $B'$ are $\nu$-measurable. 
	It is well known that the sigma algebra of Borel subsets of $[0, 1]^{n-1}$
	is generated by the closed half-spaces bounded by hyperplanes which are 
	orthogonal to the coordinate axes. Therefore it suffices to establish that for 
	all $c\in[0,1]$ and $i\in[n-1]$ the set 
	\[
		B_i(c)=\bigl\{(x_1,\dots,x_{n-1})\colon x_i\leq c\bigr\}  
	\]
	is $\nu$-measurable. This means that for each test set $E\subseteq [0,1]^{n-1}$ 
	we need to prove 
	(see \cite{Bogachev}*{Proposition 1.5.11}) that
	\begin{equation}\label{eq:Bic}
		\nu(E) = \nu\bigl(E\cap B_i(c)\bigr) + \nu\bigl(E\setminus B_i(c)\bigr) \,. 
	\end{equation}

	In case $i\in [n-2]$ this rewrites as 
	\[
		\cH^{n-1}(F_E)
		=
		\cH^{n-1}\bigl(F_E\cap B_i(c)\bigr)+\cH^{n-1}\bigl(F_E\setminus B_i(c)\bigr)
	\]
	and follows from the measurability of $B_i(c)$. Thus we may suppose $i=n-1$
	from now on. 
	
	Setting
	\begin{align*}
		F_E^{-} &= \bigl\{ (x_1,\ldots,x_{n-2}, \fh_A(x_1,\ldots,x_{n-1}))
		\colon(x_1,\ldots,x_{n-1})\in E\cap \pi_n(A) \text{ and } x_{n-1} \le c\bigr\}\,,\\
		\intertext{ and }
		F_E^{+}&= \bigl\{(x_1,\ldots,x_{n-2}, \fh_A(x_1,\ldots,x_{n-1}))\colon
		 (x_1,\ldots,x_{n-1})\in E\cap \pi_n(A) \text{ and } x_{n-1} > c\bigr\}
	\end{align*}
	we can reformulate~\eqref{eq:Bic} as 
	\begin{equation}\label{eq:LL}
		\cH^{n-1}(F_E) = \cH^{n-1}(F_E^{-}) + \cH^{n-1}(F_E^{+})\,.
	\end{equation}

	Now by Fact~\ref{fact:L} the set
	\[
		L =\bigl\{ (x_1,\ldots,x_{n-1})\in [0,1]^{n-1}\colon
			 x_{n-1}<\fh_A(x_1,\ldots,x_{n-2},c)\bigr\}
	\]
	is Lebesgue measurable, whence  
	\begin{equation}\label{eq:L-mes}
		\cH^{n-1}(F_E) = \cH^{n-1}(F_E\setminus L) + \cH^{n-1}(F_E\cap L)\,.
	\end{equation}
	Moreover, the set 
	\[
		N = \bigl\{(x_1,\ldots,x_{n-2}, \fh_A(x_1,\ldots,x_{n-2},c))
			\colon x_1,\ldots,x_{n-2}\in [0,1]\bigr\}
	\]
	is a weak antichain in $[0,1]^{n-1}$, so by Corollary~\ref{cor:N} we have $\cH^{n-1}(N)=0$. 
	
	Together with the inclusions
	\begin{align*}
	(F_E\setminus L)\setminus N &\subseteq F_E^{-} \subseteq F_E\setminus L\\
	\text{and } \qquad
	F_E \cap L &\subseteq F_E^{+} \subseteq (F_E\cap L)\cup N\,,
	\end{align*} 
	which follow from the fact that $\fh_A$ is order-reversing, this shows
	\[
		\cH^{n-1}(F_E^{-}) = \cH^{n-1}(F_E\setminus L)
		\quad\text{ and }\quad
		\cH^{n-1}(F_E^{+}) = \cH^{n-1}(F_E\cap L)\,.
	\]
	Therefore~\eqref{eq:L-mes} implies~\eqref{eq:LL}.	
\end{proof}

\subsection{Proof of Lemma~\ref{lem:gmt}}
\label{subsec:gmt}

There are two issues that need to be addressed when transferring the proof of the 
projection inequality for smooth antichains sketched in Subsection~\ref{subsec:outline}
to the general case. Starting with the representation~\eqref{eq:Bf} of a given 
antichain $A$ with an order-reversing function $f$, we need to deal with the fact that $B$
may fail to be measurable and, moreover, with the possible non-differentiability of $f$. 
It turns out that these two points can be handled separately from each other and we start 
by giving an argument that applies to the case where $f$ is linear and $B$ may be arbitrary.
 
\begin{lemma}\label{lem:linear}
	If $L\colon \RR^{n-1}\longrightarrow \RR$ is a linear function,  $B\subseteq [0, 1]^{n-1}$
	is arbitrary and 
	\[
		S=\bigl\{(\fx, L(\fx))\colon \fx\in B\bigr\}\,,
	\]
	then
	\[
		\cH^{n-1}(S)
		\le 
		\cH^{n-1}(B)+ \sum_{i=1}^{n-1}\mathcal H^{n-1}\bigl(\pi_i(S)\bigl)\,.
	\]
\end{lemma}

\begin{proof}
Let $L$ be given by 
\[
	(x_1, \ldots, x_{n-1})\longmapsto \sum_{i=1}^{n-1} c_ix_i\,.
\]

For $i\in[n-1]$ the map: $B \longrightarrow \pi_i(S)$ given by
\[
	(x_1, \ldots, x_{n-1})\longmapsto (x_1, \ldots, x_{i-1}, x_{i+1}, \ldots, x_{n-1}, 
		c_1x_1+\ldots+c_{n-1}x_{n-1})
\]
and the map: $B \longrightarrow S$ given by
\[
	(x_1, \ldots, x_{n-1})\longmapsto (x_1, \ldots, x_{n-1}, 
		c_1x_1+\ldots+c_{n-1}x_{n-1})
\]
are linear and surjective. Since 
the $(n-1)$-dimensional
Hausdorff measure of $B$ agrees with the $(n-1)$-dimensional Lebesgue outer measure of $B$,
we have (see e.g.~\cite{Evans_Gariepy}*{p.114})
$H^{n-1}\bigl(\pi_i(S)\bigl)=|c_i|\cdot \cH^{n-1}(B)$ and
$\cH^{n-1}(S)=\bigl(1+\sum_{i=1}^{n-1}c_i^2\bigr)^{1/2}\cH^{n-1}(B)$.
Thus it remains to remark 
\[
	\bigl(1+\sum_{i=1}^{n-1}c_i^2\bigr)^{1/2}
	\le 
	1+\sum_{i=1}^{n-1}|c_i|\,,
\]
which is clear. 	
\end{proof}

Recall that a function 
$f\colon F \subseteq \RR^n\longrightarrow\RR^m$ is 
\emph{Lipschitz with constant $K$} (or $K$-\emph{Lipschitz} for short) if 
\[ 
	\n{f(\fx)-f(\fy)} \leq K\cdot \n{\fx-\fy} \text{ for all }\fx,\fy\in F\,. 
\]
We use several times the following well-known result concerning the $s$-dimensional 
Hausdorff measure (see \cite{Falconer_1990}*{p.24}).

\begin{lemma}\label{lem:Lip}
	Let $m$ and $n$ be positive integers and let $F\subseteq \RR^n$. 
	If $f\colon F\longrightarrow\RR^m$ is a $K$-Lipschitz function, 
	then $\cH^s(f(F))\leq K^s \cdot\cH^s(F)$.
\end{lemma}

For the rest of this subsection we fix an antichain $A$ in $[0,1]^n$ and a positive real 
number~$\delta$ for which we would like to establish Lemma~\ref{lem:gmt}. Let $\fh_A$ be the 
function associated with~$A$ (see~\eqref{eq:fahut}). If for some $\fx\in(0, 1)^{n-1}$ 
and $i\in[n-1]$ the $i$-th partial derivative of~$\fh_A$ at~$\fx$ exists, we denote it 
by $D_i \fh_A(\fx)$. Furthermore, if a point~$\fx$ has the property that all partial 
derivatives $D_1 \fh_A(\fx), \ldots, D_{n-1} \fh_A(\fx)$ exist, we define 
$\Lx\colon\RR^{n-1}\longrightarrow \RR$ to be the linear form given by 
\[
	\Lx(v_1,\ldots, v_{n-1}) = \sum_{i=1}^{n-1} D_i \fh_A(\fx) v_i\,.
\]

The Borel set $B$ we need to exhibit will be a subset of a closed 
set $C\subseteq (0,1)^{n-1}$ on which $\fh_A$ has some useful differentiability 
properties collected in the lemma that follows. 

\goodbreak
  
\begin{lemma}\label{lem:C}
	There exists a closed set $C\subseteq (0,1)^{n-1}$ such that
	\begin{enumerate}[label=\rmlabel]
		\item\label{it:C1} $\cH^{n-1}(C) > 1-\delta/2$;
		\item\label{it:C2} all partial derivatives of $\fh_A$ exist and are continuous on $C$;
		\item\label{it:C3} for every $\fx\in C$, the function $\fh_A$ is differentiable 
			at $\fx$ with the derivative $\Lx$;
		\item\label{it:C4} the differentiability of $\fh_A$ is uniform on $C$, i.e., for 
			every $\eta>0$ there exists an~$\varepsilon>0$ such that for
			all $\fa, \fx\in C$ with $\n{\fa-\fx}<\varepsilon$ we have 
			\[
				\big|\fh_A(\fa)- \fh_A(\fx)- \Lx(\fa-\fx)\big| 
				\le 
				\eta \n{\fa-\fx}\,. 
			\]
\end{enumerate}
\end{lemma}

\begin{proof}
	Since the function $\fh_A$ is order-reversing, Lemma~\ref{lem:CC} implies that it is 
	almost everywhere differentiable. Hence there exists a measurable set 
	$C_1\subseteq (0,1)^{n-1}$, whose measure equals $1$, such that for every $\fx\in C_1$ 
	all partial derivatives of $\fh_A$ exist and, moreover,  $\fh_A$ is differentiable 
	at $\fx$ with the derivative $\Lx$. So by choosing $C\subseteq C_1$ later, 
	we can ensure~\ref{it:C3} as well as the first part of \ref{it:C2}.
	
	Next, by Lusin's theorem (see e.g.~\cite{Bogachev}*{Theorem 2.2.10}), there exists 
	a closed set $C_2\subseteq C_1$ with $\cH^{n-1}(C_2)>1-\delta/4$ such that all 
	partial derivatives $D_i\fh_A$ exist and are continuous on $C_2$. So every $C\subseteq C_2$
	satisfies~\ref{it:C2} as well. 
	
	Now define for every $m\in\NN$ the measurable function 
	$g_m\colon C_2\longrightarrow \RR$ by
	\[
		g_m(\fx) 
		= 
		\sup\biggl\{\frac{| \fh_A(\fa)- \fh_A(\fx)- \Lx(\fa-\fx)|}{\n{\fa-\fx}}
		\colon \fa\in \bigl(\QQ\cap (0,1)\bigr)^{n-1}, \, 0<\n{\fa-\fx}<1/m   \biggr\}\,.
	\]
	Since $\lim_{m \to \infty} g_m(\fx) = 0$ holds for every $\fx\in C_2$, Egoroff's theorem 
	(see e.g.~\cite{Bogachev}*{Theorem~2.2.1}) implies that there exists a closed set 
	$C\subseteq C_2$ with $\cH^{n-1}(C)>1-\delta/2$ and such that $g_m\to 0$ holds 
	uniformly on $C$. Such a set has the properties~\ref{it:C1} and~\ref{it:C4} as well.
\end{proof}

Throughout the remainder of this subsection, $C$ denotes a set provided by the 
previous lemma. Set
\begin{equation}\label{eq:K}
	K=\left(1 + \frac{\delta}{n}\right)^{1/(2n-2)}
\end{equation}
and for $\fx\in (0,1)^{n-1}$ and $\varepsilon>0$ let
\[
	\Q = \bigl\{\fy\in \RR^{n-1}\colon \| \fx-\fy \|_{\infty}<\varepsilon\bigr\} 
\]
be the $\varepsilon$-cube around $\fx$. Next we intend to show for every 
$\fx\in C$, that if $\eps>0$ is sufficiently small, then the projection inequality holds
in an approximate form for $A\cap \pi_{n}^{-1}(C \cap \Q)$ instead of $A$. Once this 
is known, a Vitali covering argument will allow us to combine many such cubes, so that 
the desired set $B$ can be taken to be a disjoint union of several sets of the form $C\cap \Q$.
The definition that follows collects some properties of such cubes that will be 
useful for implementing this strategy.  

\begin{dfn}\label{dfn:nice}
	Given $\fx\in C$ and $\eps>0$ the $\varepsilon$-cube $Q=\Q$ is said to be \emph{nice} 
	if it has the following properties:
	\begin{enumerate}[label=\alabel]
	\item\label{it:N1} $Q\subseteq [0,1]^{n-1}$.
	\item\label{it:N2} If $\fa,\fb\in Q\cap C$, then
		\[ 
			\big|\fh_A(\fa)-\fh_A(\fb)\big|^2 
			\le 
			(K^2-1) \cdot\n{\fa-\fb}^2 + K^2\cdot \big|\Lx(\fa-\fb)\big|^2\,.
		\]
	\item\label{it:N3} If $\fa,\fb\in Q\cap C$, $i\in [n-1]$, and $D_i \fh_A(\fx)\neq 0$, 
		then
		\[
			\big|\Lx(\fa-\fb)\big|^2 
			\le 
			(K^2-1)\sum_{j\in [n-1]\setminus \{i\}} |a_j-b_j|^2 
				+ K^2\cdot \big|\fh_A(\fa)-\fh_A(\fb)\big|^2\,.
		\]
\end{enumerate}
\end{dfn}   

The following result shows that nice cubes determine parts of $A$, for which the projection 
inequality holds up to a multiplicative factor that is close to $1$. 

\begin{lemma}\label{lem:use-nice}
	If $Q=\Q$ is a nice cube and $A_Q =A\cap \pi_{n}^{-1}(Q\cap C)$, then 
	\[ 
		\cH^{n-1}(A_Q) \le K^{2(n-1)}\cdot \sum_{i=1}^{n} \cH^{n-1}\bigl(\pi_i(A_Q)\bigr)\,. 
	\]
\end{lemma}

\begin{proof}
	Observe that Definition~\ref{dfn:nice}\ref{it:N2} asserts that the 
	map $\bigl(\fa,\Lx(\fa)\bigr)\longmapsto \bigl(\fa,\fh_A(\fa)\bigr)$ from the set 
	$S=\bigl\{(\fa,\Lx(\fa))\colon \fa\in \pi_n(A_Q)\bigr\}$ onto the set $A_Q$ is 
	Lipschitz with constant $K$. Therefore Lemma~\ref{lem:Lip} and 
	Lemma~\ref{lem:linear} (applied with~$\pi_n(A)$ and~$\Lx$ here in place of~$B$ and~$L$ 
	there) yield
	\begin{eqnarray*}
		\cH^{n-1}(A_Q)
		&\le& 
		K^{n-1}\cdot \cH^{n-1}(S) \\
		&\le& 	
		K^{n-1}\Bigl(\cH^{n-1}\bigl(\pi_n(A_Q)\bigr)
			+ \sum_{i=1}^{n-1}\cH^{n-1}\bigl(\pi_i(S)\bigr)\Bigr)\,.
	\end{eqnarray*}
	So to conclude the proof it suffices to show
	\begin{equation}\label{eq:use-c}
		\cH^{n-1}\bigl(\pi_i(S)\bigr) \le K^{n-1} \cH^{n-1}\bigl(\pi_i(A_Q)\bigr) 
		\quad \text{ for all } 
		i \in [n-1]\,.
	\end{equation}

	If $D_i \fh_A(\fx)=0$, the set $\pi_i(S)$ is contained in an $(n-2)$-dimensional 
	vector space and~\eqref{eq:use-c} is clear. On the other hand, if $D_i \fh_A(\fx)\ne 0$, 
	then Definition~\ref{dfn:nice}\ref{it:N3} implies that the map 
	\[
		\bigl(\fa_i,\fh_A(\fa)\bigr) \longmapsto \bigl(\fa_i,\Lx(\fa)\bigr)
	\]
	from $\pi_i(A_Q)$ to $\pi_i(S)$ is Lipschitz with constant $K$, which 
	entails~\eqref{eq:use-c} in view of Lemma~\ref{lem:Lip}.
\end{proof}

Next we show that nice cubes are ubiquitous.  

\begin{lemma}\label{lem:get-nice}
	Given $\fx\in C$, the cube $\Q$ is nice for every sufficiently small $\varepsilon >0$. 
\end{lemma}

\begin{proof}
	We verify for each of the three clauses in Defintion~\ref{dfn:nice} separately  
	that it holds for every sufficiently small $\varepsilon >0$. 
	Since $C \subseteq(0,1)^{n-1}$, this is immediate for~\ref{it:N1}.     
	For~\ref{it:N2}, we put
	\[
		\eta = \frac{K^2-1}{2K}
	\]
	and let $\varepsilon>0$ be sufficiently small. For arbitrary $\fa,\fb\in \Q\cap C$ 
	Lemma~\ref{lem:C} yields
	\begin{eqnarray*}
		\big|\fh_A(\fa)-\fh_A(\fb)\big| 
		&\overset{\text{\ref{it:C4}}}{\le}& 
		|L_{\fa}(\fa-\fb)| + \eta \n{\fa-\fb} \\
		&\le& 
		|\Lx(\fa-\fb)| + (\|\Lx-L_{\fa}\| + \eta) \n{\fa-\fb}\\
		&\overset{\text{\ref{it:C2}}}{\le}& 
		|\Lx(\fa-\fb)| + 2\eta \n{\fa-\fb}\,. 
	\end{eqnarray*}
	Now the Cauchy-Schwarz inequality implies 
	\begin{eqnarray*}
		\big|\fh_A(\fa)-\fh_A(\fb)\big|^2 
		&\le& 
		\left(\frac{1}{K}\cdot K\cdot|\Lx(\fa-\fb)| + \frac{2\eta}{\sqrt{K^2-1}}
			\cdot \sqrt{K^2-1}\cdot \n{\fa-\fb}\right)^2  \\
		&\le& 
		\left(\frac{1}{K^2} + \frac{4\eta^2}{K^2-1} \right)\cdot 
			\left( K^2\cdot |\Lx(\fa-\fb)|^2 + (K^2-1)\cdot \n{\fa-\fb}^2 \right)\,, 
	\end{eqnarray*}
	and, as the first factor is equal to $1$ by the definition of $\eta$, 
	this proves part~\ref{it:N2} of Definition~\ref{dfn:nice}.  

	It remains to check~\ref{it:N3}. Without loss of generality, we may assume that $i=n-1$ 
	and $D_{n-1}f(\fx)\neq 0$. Set
	\begin{align*}
		\lambda' &= \biggl(\sum_{i=1}^{n-2} D_i \fh_A(\fx)^2\biggr)^{1/2}\,,\\
		\lambda''&= \big|D_{n-1}\fh_A(\fx)\big|\,,\\
		\xi &= \min\left\{\frac{\lambda''}{2}, 
			\frac{\lambda''\sqrt{K^2-1}}{2(\lambda' + \lambda'')}\right\}\,,\\
		\eta &= \frac{(K-1)\xi}{2K}\,,
	\end{align*}
	let $\varepsilon>0$ be sufficiently small, and fix arbitrary points $\fa,\fb\in \Q\cap C$. 
	Recall that we have to show 
	\[
		\big|\Lx(\fa-\fb)\big|^2 
		\le 
		(K^2-1) \sum_{i=1}^{n-2} (a_i-b_i)^2 + K^2\cdot \big|\fh_A(\fa)-\fh_A(\fb)\big|^2\,.
	\]
	To this end, it suffices to establish the following two implications:
	\begin{enumerate}[label=\nlabel]
		\item\label{it:11} If $|\Lx(\fa-\fb)| \ge \xi \cdot \n{\fa-\fb}$, 
			then $|\Lx(\fa-\fb)| \le  K|\fh(\fa)-\fh(\fb)|$.
		\item\label{it:22} If $|\Lx(\fa-\fb)| \le \xi \cdot \n{\fa-\fb}$,
			then $|\Lx(\fa-\fb)| \le \sqrt{K^2-1}\cdot \n{\fa_{n-1}-\fb_{n-1}}$,
	\end{enumerate}
	where $\fa_{n-1}=(a_1,\dots,a_{n-2})$ and $\fb_{n-1}$ is defined analogously. 
 	For the proof of~\ref{it:11} we observe that, similarly as before, Lemma~\ref{lem:C} 
	yields 
	\begin{eqnarray*}
	|\fh_A(\fa)-\fh_A(\fb)| 
	&\overset{\text{\ref{it:C4}}}{\ge}& 
	|L_{\fa}(\fa-\fb)| - \eta \n{\fa-\fb} \\
	&\ge& 
	|\Lx(\fa-\fb)| - \left( \|\Lx-L_{\fa} \| + \eta \right)\cdot \n{\fa-\fb}\\
	&\overset{\text{\ref{it:C2}}}{\ge}&  
	|\Lx(\fa-\fb)| - 2\eta \n{\fa-\fb}\,.
	\end{eqnarray*}
	Moreover, the definitions of $\xi$ and $\eta$ imply
	\[ 
		2\eta \n{\fa-\fb} 
		= 
		\frac{K-1}{K}\cdot \xi\cdot \n{\fa-\fb} \le \frac{K-1}{K}\cdot |\Lx(\fa-\fb)|\,,
	\]
	so that altogether we arrive at the desired estimate
	\[
		|\fh(\fa)-\fh(\fb)| \ge \frac{1}{K}\cdot |\Lx(\fa-\fb)|\,.
	\]

	Proceeding with~\ref{it:22} we set $c_i=D_i \fh_A(\fx)$ for every $i\in [n-1]$ 
	and $\mathbf{c}=(c_1,\dots,c_{n-1})$.
	Thus $\lambda'= \n{\mathbf{c}_{n-1}}$, $\lambda''=|c_{n-1}|$, and 
	$\Lx(\fa-\fb)=\mathbf{c} \cdot (\fa-\fb)$.
	Owing to the Cauchy-Schwarz inequality, the triangle inequality, and 
	the assumption of~\ref{it:22} we have
	\begin{align*}
		\lambda''\cdot |a_{n-1}-b_{n-1}| - \lambda' \cdot \n{\fa_{n-1}-\fb_{n-1}}
		&=|c_{n-1}(a_{n-1}-b_{n-1})| - \n{\mathbf{c}_{n-1}} \n{\fa_{n-1}-\fb_{n-1}}\\
		&\le |c_{n-1}(a_{n-1}-b_{n-1})| - |\mathbf{c}_{n-1} \cdot (\fa_{n-1}-\fb_{n-1})|\\
		&\le |\mathbf{c} \cdot (\fa-\fb)|\\
		&\le \xi \cdot \n{\fa-\fb}\\
		&\le \xi\cdot |a_{n-1}-b_{n-1}| + \xi\cdot \n{\fa_{n-1}-\fb_{n-1}}\,.
	\end{align*}
	Since $\xi\le \frac{\lambda''}{2}$, this leads to 
	\[
		\lambda''\cdot |a_{n-1}-b_{n-1}| 
		\le 
		(2\lambda' +\lambda'') \cdot \n{\fa_{n-1}-\fb_{n-1}}\,,
	\]
	wherefore 
	\[
		\n{\fa-\fb}
		\le 
		\n{\fa_{n-1}-\fb_{n-1}} + |a_{n-1}-b_{n-1}|
		\le 
		\frac{2(\lambda'+\lambda'')}{\lambda''} \cdot \n{\fa_{n-1}-\fb_{n-1}}\,.
	\]

	Hence we have indeed
	\[ 
		|\Lx(\fa-\fb)| \le \xi\cdot \n{\fa-\fb}
		 \le 
		 \sqrt{K^2-1}\cdot \n{\fa_{n-1}-\fb_{n-1}}\,,
	\]
	which concludes the proof.
\end{proof}

Given a measurable set $S\subseteq \RR^d$, we shall say that a family $\cV$ of open 
$d$-dimensional cubes forms a {\it Vitali covering} of $S$ if for every $\fx\in S$ 
and $\eps>0$ there is a cube $Q\in\cV$ with~$\fx\in Q$ and $\diam(Q)<\eps$.
Recall that by {\it Vitali's Covering Theorem} (see e.g.~\cite{Cohn}*{p.164})
in such a situation there is a countable subset $\cU\subseteq \cV$ such that 
the members of $\cU$ are mutually disjoint and $S\setminus \bigcup\cU$ is null
with respect to the Lebesgue measure.  

\begin{proof}[Proof of Lemma~\ref{lem:gmt}]
	By Lemma~\ref{lem:get-nice} the collection of nice cubes forms a Vitali covering 
	of~$C$. Therefore Vitali's Covering Theorem yields countably many 
	mutually disjoint nice cubes that cover $C$ except for a null set. In view
	of the compactness of $C$ and Lemma~\ref{lem:C}\ref{it:C1} this leads to finitely
	many mutually disjoint nice cubes, say  $Q^{(1)},\ldots,Q^{(N)}$, such that the 
	Borel set
	\[
		B = \bigcup_{k\in [N]} \bigl(C\cap Q^{(k)}\bigr)
	\]
	satisfies $\cH^{n-1}(B)>1-\delta$.
	It remains to show 
	\begin{equation}\label{eq:313}
		\cH^{n-1}(A')\le \sum_{i=1}^{n}\cH^{n-1}\bigl(\pi_i(A')\bigr)+\delta\,,
	\end{equation}
	where $A'=A\cap\pi_{n}^{-1}(B)$.

	Setting 
	\[
		A^{(k)} = A\cap \pi_{n}^{-1} \bigl(C\cap Q^{(k)}\bigr)
	\]
	for every $k\in [N]$, we infer
	\begin{equation}\label{eq:5}
		\cH^{n-1}(A^{(k)}) 
		\le 
		K^{2(n-1)} \cdot\sum_{i=1}^{n} \cH^{n-1}\bigl(\pi_i(A^{(k)})\bigr)
	\end{equation}
	from Lemma~\ref{lem:use-nice}.

	Next we observe that due to 
	\[
		A' = \bigcup_{k\in [N]} A^{(k)}
	\]
	a repeated application of Lemma~\ref{lem:decomp} reveals 
	\[
		\sum_{k=1}^{N} \cH^{n-1}\bigl(\pi_i(A^{(k)})\bigr) = \cH^{n-1}\bigl(\pi_i(A')\bigr)
	\]
	for every $i\in [n]$. Therefore, by summing~\eqref{eq:5} over $k$ 
	we obtain
	\begin{eqnarray*}
		\cH^{n-1}(A') 
		&\le& 
		\sum_{k=1}^{N} \cH^{n-1}\bigl(A^{(k)}\bigr) \\
		&\le& 
		K^{2(n-1)}\cdot\sum_{i=1}^{n} \cH^{n-1}\bigl(\pi_i(A')\bigr) \\ 
		&\le& 
		\sum_{i=1}^{n} \cH^{n-1}\bigl(\pi_i(A')\bigr) + n\bigl(K^{2(n-1)}-1\bigr)\\
	\end{eqnarray*}
	and our choice of $K$ in~\eqref{eq:K} leads to the desired estimate~\eqref{eq:313}.
\end{proof}
 
\section{Proof of the main Theorem}\label{sec:weak}

Let $n\ge 2$ be fixed throughout this section. 
We begin by describing a construction that allows us to ``approximate'' a given $n$-dimensional 
weak antichain with arbitrary ``accuracy'' by an antichain.  

It will be convenient to write $\Sx=\sum_{i=1}^n x_i$ for $\fx \in \RR^n$. 
Moreover, for every $\eps\in \bigl(0, \frac 1{2n}\bigr)$ we let 
$f_{\epsilon}\colon  \RR^n \rightarrow \RR^n$ denote the linear transformation
\[
	(x_1,\dots,x_n) \longmapsto \bigl(x_1-\eps\Sx,\ldots,x_n-\eps\Sx\bigr)\,,
\]
and set 
\begin{equation}\label{eq:Leps}
	L_\eps=\frac{1}{\sqrt{1-2n\eps}}\,.
\end{equation}

One checks easily that
\begin{equation}\label{eq:Sfx}
	S(f_{\epsilon}(\fx))=(1-n\varepsilon)\Sx
\end{equation}
and, consequently, $f_\eps$ is invertible. Let us also note that $f_\eps$ maps 
$[0, 1]^{n}$ into $[-1, 1]^{n}$.

\begin{fact}\label{fact:Leps}
	Let $\eps\in \bigl(0, \frac 1{2n}\bigr)$.
	\begin{enumerate}[label=\rmlabel]
		\item\label{it:eps1} The inverse $f_\eps^{-1}$ is $L_\eps$-Lipschitz.
		\item\label{it:eps2} If $A\subseteq [0, 1]^n$ is a weak antichain, then 
			$f_{\epsilon}(A)$ is an antichain.
	\end{enumerate}
\end{fact}

\begin{proof}
	For part~\ref{it:eps1} it is enough to verify 
	\[
		\n{f_{\epsilon}(\fx)-f_{\epsilon}(\fy)}^2 \ge (1-2n\eps)\n{\fx-\fy}^2
	\]
	for any two points $\fx, \fy\in\RR^n$. In terms of $\fa=\fx-\fy$ this rewrites as
	\[
		\sum_{i=1}^n \bigl(a_i-\eps \Sa\bigr)^2 \ge (1-2n\eps) \sum_{i=1}^n a_i^2\,,
	\]
	i.e.,
	\[
		2n\eps \sum_{i=1}^n a_i^2  \ge (2\eps-n\eps^2) \Sa^2\,.
	\]

	This follows from the fact that the Cauchy-Schwarz inequality yields the 
	even stronger estimate
	\[
		n\cdot \sum_{i=1}^n a_i^2 \ge \Sa^2\,.
	\]

	Now assume that contrary to~\ref{it:eps2} we have a weak antichain $A\subseteq [0, 1]^n$
	and two distinct points $f_{\epsilon}(\fx), f_{\epsilon}(\fy) \in f_{\epsilon}(A)$ with 
	$f_{\epsilon}(\fx) \le f_{\epsilon}(\fy)$. Using 
	$S(f_{\epsilon}(\fx)) < S(f_{\epsilon}(\fy))$
	and~\eqref{eq:Sfx} we obtain $\Sx < \Sy$. So for every $i \in [n]$ the assumption
	\[
		x_i-\eps\Sx \le y_i -\eps\Sy
	\]
	yields 
	\[
		x_i \le y_i+\varepsilon(\Sx-\Sy) < y_i\,.
	\]
	But~$\fx\leftY\fy$ contradicts $A$ being a weak antichain.
\end{proof}

Later it will be useful to know that in the situation of Fact~\ref{fact:Leps}\ref{it:eps2}
the projection inequality (as in Proposition~\ref{prop:anti}) applies 
to $f_\eps(A)\subseteq [-1, 1]^{n}$. This is because the homothety from~$[-1, 1]^{n}$
onto $[0, 1]^{n}$ sends $f_\eps(A)\subseteq [-1, 1]^{n}$ onto an antichain 
in~$[0, 1]^{n}$ and the $\cH^{n-1}$-measure gets rescaled by a factor of $2^{n-1}$
under this map. 
 
%
\begin{proof}[Proof of Theorem~\ref{thm:main}]
	We divide the argument into three steps.

	{\bf Part I.} Suppose first that $A$ is a compact set. 
	Let $\eps\in\bigl(0, \frac 1{2n}\bigr)$ be arbitrary and recall that by 
	Fact~\ref{fact:Leps}\ref{it:eps2} the projection inequality 
	applies to $f_\eps(A)$. In combination with Lemma~\ref{lem:Lip} and 
	Fact~\ref{fact:Leps}\ref{it:eps1} we obtain
	\[
		\cH^{n-1}(A)   
		\le  
		L_\eps^{n-1}\cdot \cH^{n-1}\bigl(f_{\eps}(A)\bigr) 
		\le 
		L_\eps^{n-1}\cdot \sum_{i=1}^n\cH^{n-1}\bigl(\pi_i(f_{\eps}(A))\bigr)\,.
	\]
	Therefore it suffices to prove for every $i\in[n]$ that
	\begin{equation}\label{eq:ww}
		\liminf_{\eps\to 0} \cH^{n-1}\bigl(\pi_i(f_{\eps}(A))\bigr) 
		\le 
		\cH^{n-1}\bigl(\pi_i(A)\bigr)\,.
	\end{equation}

	Fix $i\in [n]$. Since $\|\pi_i(f_{\varepsilon}(\fa))- \pi_i(\fa)\| \le n^2\varepsilon$ 
	holds for all $\fa\in A$ and $\eps>0$, we have 
	\[
			\pi_i\bigl(f_{\epsilon}(A)\bigr)\subseteq \pi_i(A)^{(n^2\varepsilon)}\,,
	\]
	where the notation is as in~\eqref{eq:Sdelta}.
	Thus a complementary variant of Fact~\ref{fact:haus-mono} yields
	\[
		\liminf_{\eps\to 0} \cH^{n-1}\bigl(\pi_i(f_{\eps}(A))\bigr)
		\le
		\liminf_{\eps\to 0} \cH^{n-1}\bigl(\pi_i(A)^{(n^2\eps)}\bigr)
		\le 
		\cH^{n-1}\left(\bigcap\nolimits_{\varepsilon>0}\pi_i(A)^{(n^2\varepsilon)}\right)\,.
	\]

	As the compactness of $A$ implies 
	$\bigcap_{\varepsilon>0}\pi_i(A)^{(n^2\varepsilon)}=\pi_i(A)$,
	we thereby arrive at~\eqref{eq:ww}.
	
	\medskip
	
	{\bf Part II.} Next we treat the more general case that $A$ 
	is $\cH^{n-1}$-measurable. For every compact $K\subseteq A\cap (0, 1)^n$
	the result of the first part entails
	\begin{equation}\label{eq:KA}
		\cH^{n-1}(K)
		\le 
		\sum_{i=1}^n \cH^{n-1}\bigl(\pi_i(K)\bigr)
		\le 
		\sum_{i=1}^n \cH^{n-1}\bigl(\pi_i(A)\bigr)\,.
	\end{equation}

	Since $\cH^{n-1}(A)$ is finite by Corollary~\ref{cor:N},
	the measurability of $A$ implies (see e.g.~\cite{Evans_Gariepy}*{Theorem 1.7 
	and Theorem 1.8(ii)}) that 
	\[
		\cH^{n-1}(A)=\sup\bigl\{\cH^{n-1}(K)\colon 
			\text{$K$ is a compact subset of $A$}\bigr\}\,.
	\]

	So the desired result follows from~\eqref{eq:KA}.
	
	\medskip
	 
    {\bf Part III. } Assume finally that $A$ is an arbitrary weak antichain. 
    We claim that the closure $\bar{A}$ of $A$ is likewise a weak antichain. 
    Otherwise there existed two distinct points $\fa,\fb\in\bar{A}$ such that 
    $\fa \leftY \fb$. Observe that there are sufficiently small neighbourhoods  
    $U(\fa)$ and $U(\fb)$ of $\fa$ and $\fb$ respectively, 
    such that $\mathbf{c} \leftY \mathbf{d}$ holds for all $\mathbf{c}\in U(\fa)$ 
    and all $\mathbf{d}\in U(\fb)$. Since $U(\fa)$ and $U(\fb)$ necessarily contain points 
    of $A$, we get a contradiction to the fact that $A$ is a weak antichain. 
	This proves that $\bar{A}$ is indeed a weak antichain. 
	
	Now by the Borel regularity of the Hausdorff measure, there exists for 
	every $i\in [n]$ a Borel set $B_i\supseteq\pi_i(A)$ such that 
	$\mathcal H^{n-1}(\pi_i(A))=\mathcal H^{n-1}(B_i)$. Applying the result of the 
	previous step to the measurable weak antichain 
	\[
		A^\star = \bar{A}\cap\bigcap_{i=1}^n \pi_i^{-1}(B_i)\,,
	\]
    we obtain
	\begin{eqnarray*}
		\cH^{n-1}(A) 
		&\le& 
		\cH^{n-1}(A^{\star}) 
		\le
		\sum_{i=1}^n \cH^{n-1}\bigl(\pi_i(A^\star)\bigr) \\
		&\le& 
		\sum_{i=1}^n \cH^{n-1}(B_i)  
		= 
		\sum_{i=1}^n \cH^{n-1}\bigl(\pi_i(A)\bigr)\,,
	\end{eqnarray*}
	as required. 
\end{proof}

\section{Concluding remarks}\label{sec:conc}

\subsection{The discrete case}

For $n\ge 2$ the only case where equality holds in Theorem~\ref{thm:discrete}
occurs when $A=\varnothing$. This gives rise to the following question. 

\begin{problem}
	Determine an optimal lower bound $g(n, m)$ on the gap $\sum_{i=1}^n|\pi_i(A)|-|A|$
	as $A$ varies over weak antichains in $\ZZ^n$ of size $m$.
\end{problem}

Let us mention that a slightly improved version of the argument presented in 
Section~\ref{sec:discrete} yields 
\[
	g(n, m)\ge n-1
\]
for all integers $m, n\ge 1$.  This can be easily seen using
$|A_i| < |\pi_i(A)|$ if $A_i=\emptyset$ and $\pi_i(A_{i-1}) \neq \pi_i(A_{i})$ and consequently $|\pi_i(A_i)| < |\pi_i(A)|$ if $A_i\neq \emptyset$ and $i \ge 2$.

\subsection{Supremum vs. Maximum}

In connection with Corollary~\ref{cor:sup} one may wonder for which 
values of $n$ the supremum is attained. For instance for $n=1$ any one-point 
antichain in $[0, 1]$ has $\cH^0$-measure $1$ and a more sophisticated construction 
mentioned below shows that for $n=2$ the supremum is a maximum as well. We believe that
actually such antichains exist in all dimensions.   

\begin{conjecture}
	\label{conj:sup-max}
	For every $n \ge 1$ there is an antichain in $[0,1]^n$ whose $(n-1)$-dimensional 
	Hausdorff measure equals $n$.
\end{conjecture}

The aforementioned planar example exploits the known fact (see~\cite{Foran}*{p. 810}) 
that the $1$-dimensional Hausdorff measure of the graph of a decreasing function 
$f\colon [0,1]\longrightarrow [0,1]$ is at most $2$ and that this bound is attained 
by \emph{singular functions} (i.e., strictly decreasing functions whose derivative equals 
zero almost everywhere). 
It remains to observe that the graph of a singular function is an antichain in $[0,1]^2$.

\subsection{Skewed projections of weak antichains}
\label{subsec:skew}
So far we focused on inequalities for the orthogonal projections of weak antichains,
but we believe that, actually, more general statements hold. In order to be more precise, 
we need some additional notation.   

Given $A \subseteq [0,1]^n$ and $i \in [n]$, we set
\[
	\underline{A}_i 
	= 
	A\cap \bigl\{(x_1,\ldots,x_n)\in [0,1]^n\colon x_i = \min\{x_1,\ldots,x_n\}\bigr\}\,.
\]
Moreover, let for $i \in [n]$ the \emph{skewed projections}
$\Delta_{i}\colon \underline{A}_i \longrightarrow [0,1]^{n-1}$ be defined by 
\[
	(x_1,\dots,x_n) \longmapsto (x_1-x_i,\dots,x_{i-1}-x_i,x_{i+1}-x_i,\dots,x_n-x_i)\,.
\]
Notice that the skewed projections restricted to a weak antichain are injective. 

\begin{conjecture}\label{conj:diag}
	If $A\subseteq [0,1]^n$ is a weak antichain, then 
	\[
		\cH^{n-1}(A)\leq \sum_{i=1}^{n} \cH^{n-1}\bigl(\Delta_{i}(\underline{A}_i)\bigr)\,.
	\]
\end{conjecture}

Let us note that, if true, this would furnish a different proof of $\cH^{n-1}(A)\le n$.
As a final result, we verify the validity of this conjecture when $n=2$. 

\begin{theorem}
	Conjecture~\ref{conj:diag} is true when $n=2$. 
\end{theorem}

\begin{proof} 
	Since $\underline{A}_1$ is a weak antichain, it follows that $\Delta_1$ 
	is a bijection from $\underline{A}_i$ onto its image. Given two numbers
	$a,b\in \Delta_1(\underline{A}_1)$ with $a<b$ their inverse images under $\Delta_1$
	are of the form $\Delta_1^{-1}(a) = (x,y)$ and $\Delta_1^{-1}(b) = (x-\delta, y+\eps)$ 
	for some  $x,y\in [0,1]$ and $\delta, \eps\ge 0$. Owing to 
	\[ 
		\|\Delta_1^{-1}(a)- \Delta_1^{-1}(b)\| 
		= 
		\sqrt{\delta^2 + \eps^2} 
		\le
		\delta+\eps 
		= 
		|a-b|
	\]
	the map $\Delta_1^{-1}$ is Lipschitz with constant $1$ and Lemma~\ref{lem:Lip} 
	yields
	\[
		\cH^1(\underline{A}_1) \le \cH^1\bigl(\Delta_1(\underline{A}_1)\bigr) \,.
	\]
	Applying the same reasoning to $\Delta_2\colon\underline{A}_2\longrightarrow[0,1]$ 
	we infer
	\[
		\cH^1(\underline{A}_2)\le\cH^1(\Delta_2\bigl(\underline{A}_2)\bigr) \,.
	\]
	From these two inequalities we conclude 
	\[
		\cH^1(A)\le\cH^1\bigl(\Delta_1(A_1)\bigr)+\cH^1\bigl(\Delta_2(A_2)\bigr) \,, 
	\]
	as required. 
\end{proof}

\end{document}